\documentclass[reqno]{amsart}
\usepackage{mathrsfs}
\usepackage{hyperref}
\usepackage{float}
\usepackage{graphicx}

\usepackage{enumerate}
\newtheorem{theorem}{Theorem}[section]
\newtheorem{lemma}[theorem]{Lemma}

\theoremstyle{definition}
\newtheorem{definition}[theorem]{Definition}

\newtheorem{remark}[theorem]{Remark}

\newcommand{\norm}[1]{\left\Vert#1\right\Vert}

\numberwithin{equation}{section}
\usepackage{amsfonts}

\usepackage{amsmath}
\usepackage{amssymb}
\usepackage{cite}
\usepackage{hyperref}
\hypersetup{
   colorlinks,
    citecolor=blue,
    filecolor=black,
    linkcolor=blue,
    urlcolor=magenta
}

\begin{document}
\font\nho=cmr10
\def\dive{\mathrm{div}}
\def\cal{\mathcal}
\def\L{\cal L}

\def \ud{\underline }
\def\id{{\indent }}
\def\f{\frac}
\def\non{{\noindent}}
 \def\le{\leqslant} 
 \def\leq{\leqslant}
 \def\geq{\geqslant} 
\def\rar{\rightarrow}
\def\Rar{\Rightarrow}
\def\ti{\times}
\def\i{\mathbb I}
\def\j{\mathbb J}
\def\si{\sigma}
\def\Ga{\Gamma}
\def\ga{\gamma}
\def\ld{{\lambda}}
\def\Si{\Psi}
\def\f{\mathbf F}
\def\r{\hro{R}}
\def\e{\cal{E}}
\def\B{\cal B}
\def\A{\mathcal{A}}
\def\p{\mathbb P}

\def\tet{\theta}
\def\Tet{\Theta}
\def\hro{\mathbb}
\def\ho{\mathcal}
\def\P{\ho P}
\def\E{\mathcal{E}}
\def\n{\mathbb{N}}
\def\M{\mathbb{M}}
\def\dMu{\mathbf{U}}
\def\dMcs{\mathbf{C}}
\def\dMcu{\mathbf{C^u}}
\def\vk{\vskip 0.2cm}
\def\td{\Leftrightarrow}
\def\df{\frac}
\def\Wei{\mathrm{We}}
\def\Rey{\mathrm{Re}}
\def\s{\mathbb S}
\def\l{\mathcal{L}}
\def\C+{C_+([t_0,\infty))}
\def\o{\cal O}

\title[PAP-mild Solutions of Keller-Segel systems]{On pseudo almost Periodic Solutions of the parabolic-elliptic Keller-Segel systems}

\author[N.T. Van]{Nguyen Thi Van}\address{Nguyen Thi Van\hfill\break
Faculty of Computer Science and Engineering, Thuyloi University,  \hfill\break
175 Tay Son, Dong Da, Hanoi, Viet Nam}
\email{van@tlu.edu.vn}

\author[T.M. Nguyet]{Tran Minh Nguyet}
\address{Tran Minh Nguyet \hfill\break Department of Mathematics, Thang Long University, Nghiem Xuan Yem, Hanoi, Vietnam}
\email{nguyettm@thanglong.edu.vn}

\author[N.T. Loan]{Nguyen Thi Loan}
\address{Nguyen Thi Loan\hfill\break Faculty of Fundamental Sciences, Phenikaa University, Hanoi 12116, Vietnam.}
\email{loan.nguyenthi2@phenikaa-uni.edu.vn}

\author[P.T. Xuan]{Pham Truong Xuan}
\address{Pham Truong Xuan \hfill\break Thang Long Institute of Mathematics and Applied Sciences (TIMAS), Thang Long University, Nghiem Xuan Yem, Hanoi, Vietnam} 
\email{phamtruongxuan.k5@gmail.com or xuanpt@thanglong.edu.vn}

\begin{abstract}  
In this paper we investigate the existence, uniqueness and exponential stability of pseudo almost periodic (PAP-) mild solutions of the parabolic-elliptic (P-E) Keller-Segel system on a bounded domain $\Omega\in \mathbb{R}^n$ with smooth boundary. First, the well-posedness of the corresponding linear system is established by using the smoothing estimates of the Neumann heat semigroup on $\Omega$. Then, the existence of PAP-mild solution of linear system is done by proving a Massera-type principle. Next, we obtain the well-posedness of such solutions for semilinear system by using the results of linear system and fixed point arguments. The exponential stability is proven by using again the estimates of the Neumann heat semigroup. Finally, we discuss also such results for the case of the Keller-Segel system on the framework of real hyperbolic manifolds.
\end{abstract}



\maketitle

\tableofcontents

\section{Introduction}
In the present paper we consider the parabolic-elliptic (P-E) Keller–Segel system on a bounded domain with smooth boundary $\Omega \subset \mathbb{R}^n \,\, (n \geqslant 2)$ described by the following equations
\begin{equation}\label{KSH0} 
\left\{
  \begin{array}{rll}
u_t \!\! &= \Delta u - \chi \nabla \cdot (u\nabla v) + g(t,x) \quad  & (t,x)\in \r\times \Omega, \hfill \cr
-\Delta v +  \gamma v \!\!&= \kappa u \quad & (t,x) \in \r\times \Omega,\cr
\nabla u \cdot \nu \!\!&= \nabla v\cdot \nu =0 \quad & (t,x) \in \r\times \partial\Omega,
\end{array}\right.
\end{equation}
where $\nu$ is the normal outer vector
on $\partial\Omega$, the operator $\Delta$ means Laplace operator on $\mathbb{R}^n$ and $g:\mathbb{R}\times \Omega \to \mathbb{R}_+$ is a given function. The unknown functions $u(t,x): \mathbb{R}\times \Omega \to \mathbb{R}_+$ represents the density of cells and $v(t,x): \mathbb{R}\times \Omega \to \mathbb{R}_+$ is the concentration of the chemoattractant. The parameter $\chi$ is the sensitivity parameter, which is a positive constant. The parameters $\gamma \geqslant 0$ and $\kappa >0$ denote the decay and production rate of the chemoattractant, respectively.

Concerning the Keller-Segel (P-E) system on $\mathbb{R}^2$, the authors in \cite{Bla} proved that there exists a threshold value for the initial mass $M = \int u_0dx$ that relates to the existence and blow up of solutions: if {$M<8\pi/\kappa\chi$}, then solutions exist globally  and if {$M>8\pi/\kappa\chi$}, then solutions blow up in a finite time. For a bounded domain with smooth boundary in $\mathbb{R}^2$, Li and Wang established the finite-time blow-up and boundedness for system \eqref{KS} in \cite{Li2023}. On the other hand for the hyperbolic space $\mathbb{H}^2$, Pierfelice and Maheux obtained the local and global well-posedness results under the sub-critical condition and a blow-up result in \cite{MaPi}.

We recall briefly some well-posedness results of system \eqref{KS} on Euclidean space $\mathbb{R}^n$ (where $n \geq 3)$. In \cite{Ko1}, Kozono et al. proved the existence of weak solutions $u\in C_b(\r_+,\, L^{n/2}(\mathbb{R}^n)) \cap L^q(\r_+,\, L^p(\mathbb{R}^n))$ with the initial data $u_0 \in L^{n/2}(\mathbb{R}^n)$ small enough, where $n\geq 3$ and $p,q$ are chosen suitable. After that, Kozono and Sugiyama \cite{Ko2} proved well-posedness of mild solutions for \eqref{KS} in weak-Lorentz spacces. Precisely, for small enough initial data $u_0\in L^{n/2, \infty}(\mathbb{R}^n)\, (n \geq 3)$, they proved that system \eqref{KS} has a unique mild solution $u\in C_b(\r_+, L^{n/2,\infty}(\mathbb{R}^n))$ satisfying that $t^{1-\frac{n}{2q}} u \in C_b(\r_+,\, L^q(\r^n))$, where $n/2<q<n$. For the uniqueness in a large space of initial data, Ferreira \cite{Fe2021} established the unconditional uniqueness of mild solutions for system \eqref{KS} in $L^{p,\infty}(\mathbb{R}^n)$ spaces for {$p=n/2$}. {In addition}, for other critical spaces we can refer some works for existence and uniqueness of mild solutions of (parabolic-parabolic or parabolic-eliptic) Keller-Segel systems in \cite{Iwa2011,Fe2011,Chen2018}. The existence, uniqueness and stability of periodic and almost periodic mild solutions for Keller-Segel systems on the whole space $\mathbb{R}^n$ (where $n \geq 4$) have been established by Xuan et al. in \cite{Xu,XL2024}.

Now we present some related works on the parabolic-parabolic Keller-Segel and some other systems consisting Keller-Segel equations on bounded domain (with smooth boundary) of $\mathbb{R}^n$. The work of Winkler \cite{Winker} provided the dispersive and smoothing estimates for the Neumann heat semigroup, then employ these estimates to prove the stability of solutions for the case $n \geq 3$. 
Then, Winkler studied finite-time blow-up in the parabolic-parabolic Keller–Segel system
in the higher-dimensional case in \cite{Win2}. In \cite{Cao}, Cao extended the previous estimates obtained in \cite{Winker} to study the smallness condition on the initial data in optimal Lebesgue spaces which ensure global boundedness and large time convergence for the case $n\geq 2$. After that, Hao et al. \cite{Hao} provided  the global classical solutions to the Keller–Segel–Navier–Stokes system with matrix-valued sensitivity. In addition, Jiang proven the global Stability of Keller–Segel Systems in Critical Lebesgue
Spaces in \cite{Jiang2018}, then considered the global stability of homogeneous steady states in scaling invariant spaces for a Keller–Segel–Navier–Stokes system in \cite{Jiang2020}. We refers some useful works \cite{Win1,Win3}. In our knowledge, there is no work which studies the well-posedness of pseudo almost periodic mild solutions for the (P-E) Keller-Segel system \eqref{KSH0} on the bounded domains of $\mathbb{R}^n$, meanwhile such solution and its generalisations were been studied extensively for other parabolic and hyperbolic equations (see for example \cite{Dia} and references therein).

In the present paper, we study the existence, uniquness and exponential stability of pseudo almost periodic (PAP-) mild solutions for Keller-Segel system \eqref{KSH0}. We describle the strategy as follows: first, we employ the smoothing estimates for Neumann heat semigroup provided in \cite{Winker,Cao} to prove the well-posedness of mild solutions for the linear systems corresponding to \eqref{KSH0} (see Theorem \ref{Thm:linear} $(i)$). Base on this well-posedness we define the solution operator associating with the linear system. Then, we prove the well-posedness of PAP-mild solutions for linear systems by proving that the solution operator preserves the pseudo almost periodicity of given functions (see Theorem \ref{Thm:linear} $(ii)$). By using the well-posedness of linear systems and fixed point arguments we obtain the well-posedness of PAP-mild solutions for Keller-Segel system. The exponential stability of such solutions are also proven by using again the smoothing estimates of Neumann heat semigroup (see Theorem \ref{thm2.20}).

Finally, we will discuss about the case of Keller-Segel (P-E) system on whole hyperbolic spaces. Namely, we do not restrict the study on a bounded domain as previous sections. On the framework of hyperbolic spaces, the periodic mild solutions for the (P-E) Keller-Segel system have been treated detailed in  \cite{Xu}.
We refer some related works on the well-posedness of asymptotically almost periodic mild solutions for Navier-Stokes equations in \cite{XVQ2023,XV2023}.  In fact, we will get the same results such as the case of bounded domain in $\mathbb{R}^n$ since the scalar heat semigroup on $\mathbb{H}^n$ is also exponential stable (see Lemma \ref{dispersive} below). Using this fact, we can establish the well-posedness and exponential stability of pseudo almost periodic mild solutions for the (P-E) Keller-Segel system by the same way as in the case of bounded domains of Euclidean spaces (see Theorem \ref{ThecaseHn}).
The results obtained in this paper provides a similarity comparison of well-posedness for Keller-Segel systems on bounded domains of $\mathbb{R}^n$ and on the whole hyperbolic space $\mathbb{H}^n$. 

Our paper is organized as follows: Section \ref{2_3} relies on the (P-E) Keller-Segel systems , some useful estimates of the Neumann heat semigroup and concepts of generalized functions. In Section \ref{S3}, we provide the well-posedness of PAP-mild solutions for corresponding linear systems. In Section \ref{S4} we establish the well-posedness and exponential stablity for the (P-E) Keller-Segel system. We treat the case of systems on the hyperbolic spaces in Section \ref{S5}.

\section{The Keller-Segel (P-E) systems and concepts of functions} \label{2_3}
For simplicity, we consider that $\chi=\kappa=1$ and $g(t)= \dive f(t)$. The Keller-Segel (P-E) system \eqref{KSH0} on the bounded domain with smooth boundary $\Omega\subset \mathbb{R}^n$ (where $n \geqslant 2$) becomes 
\begin{equation}\label{KS} 
\left\{
  \begin{array}{rll}
u_t \!\! &= \Delta u - \nabla \cdot (u\nabla v) + \dive f(t) \quad  & (t,x)\in \r\times \Omega, \hfill \cr
-\Delta v  + \gamma v \!\!&= u \quad & (t,x) \in \r\times \Omega,\cr
\nabla u \cdot \nu \!\!&= \nabla v\cdot \nu =0 \quad & (t,x) \in \r\times \partial\Omega.
\end{array}\right.
\end{equation}

The second equation of system \eqref{KS} leads to $v= (-\Delta+ \gamma I)^{-1}u$. Therefore, according to Duhamel’s principle, we can define the mild solution of system \eqref{KS} on the whole line time-axis as a bounded solution of the following integral equation (see \cite{KoNa} for the same definition for Navier-Stokes equations):
\begin{equation}\label{00inte}
u(t) = \int_{-\infty}^t \nabla \cdot e^{(t-s)\Delta}\left[ - u\nabla(-\Delta + \gamma I)^{-1}u + f \right](s)ds.
\end{equation}

We recall the dispersive and smoothing estimates of Neumann heat semigroup on the bounded domain with smooth boundary $\Omega$.
\begin{lemma}\label{Heatestimates}
Suppose $(e^{t\Delta })_{t>0}$ is the Neumann heat semigroup in $\Omega$, and let $\lambda_1 > 0$ denote the first nonzero eigenvalue of $ - \Delta$ 
in $\Omega$ under Neumann boundary conditions. Then there exist positive constants $k_1,k_2,k_3, k_4$ which only depend on $\Omega$ and we have the following estimates

\begin{itemize}
\item[$(i)$] If $1 \leq q \leq p\leq \infty$, then 
\begin{equation}\label{dispersive1}
\left\| e^{t {\Delta}}\omega\right\|_{L^p(\Omega)} \leq k_1(1+t^{- \frac{n}{2}(\frac{1}{q}-\frac{1}{p})}) e^{-\lambda_1t}\left\|\omega\right\|_{L^q(\Omega)} \text{for all }  t>0
\end{equation}
 holds for all $\omega \in L^q(\Omega)$ with $\int_{\Omega}\omega dx = 0$.

\item[$(ii)$] If $1 \leq q \leq p\leq \infty$, then 
\begin{equation}\label{dispersive2}
\left\| \nabla e^{t {\Delta}}\omega\right\|_{L^p(\Omega)} \leq k_2(1+t^{-\frac{1}{2}- \frac{n}{2}(\frac{1}{q}-\frac{1}{p})}) e^{-\lambda_1t}\left\|\omega\right\|_{L^q(\Omega)} \text{for all }  t>0
\end{equation}
 holds for all $\omega \in L^q(\Omega)$.


\end{itemize}
\end{lemma}
\begin{proof}
The proof was given in \cite[Lemma 2.1]{Cao} and \cite[Lemma 1.3]{Winker}.
\end{proof}

Now we recall some concepts of generalized functions. For more details we refer the readers to books \cite{Dia} and references therein.
Let $X$ be a Banach space, we denote 
$$C_b(\mathbb{R}, X):=\{f:\mathbb{R} \to X \mid f\hbox{ is continuous on $\mathbb{R}$  and }\sup_{t\in\mathbb{R}}\|f(t)\|_X<\infty\}$$
which is a Banach space endowed with the norm $\|f\|_{\infty, X}=\|f\|_{C_b(\mathbb{R}, X)}:=\sup\limits_{t\in\mathbb{R}}\|f(t)\|_X$.

\begin{definition}(AP-function)
A function  $h \in C_b(\mathbb{R}, X )$ is called almost periodic function if for each $ \epsilon  > 0$, there exists $l_{\epsilon}>0 $ such that every interval of length $l_{\epsilon}$ contains at least a number $T $ with the following property
\begin{equation*}
 \sup_{t \in \mathbb{R} } \| h(t+T)  - h(t) \| < \epsilon.
\end{equation*}
The collection of all almost periodic functions $h:\mathbb{R} \to X $ will be denoted by $AP(\mathbb{R},X)$ which is a Banach space endowed with the norm $\|h\|_{ AP(\mathbb{R},X)}=\sup\limits_{t\in\mathbb{R}}\|h(t)\|_X.$
\end{definition}

\begin{definition}\label{WPAAfunction}(PAP-function)
A function $f \in C_b(\mathbb{R},X)$ is called pseudo almost periodic if it can be decomposed as $f = g + \phi$ where $g \in AP(\mathbb{R},X)$ and $\phi$ is a bounded continuous function with vanishing mean value i.e
\begin{equation*}\label{meanvalue}
\lim_{L\to \infty}\frac{1}{2L}\int_{-L}^L \|\phi(t)\|_X dt =0.
\end{equation*}
We denote the set of all functions with vanishing mean value by $PAP_0(\mathbb{R},X)$
and the set of all the pseudo almost periodic (PAP-) functions by $PAP(\mathbb{R},X)$. 
\end{definition}
We have that $(PAP(\mathbb{R},X),\|.\|_{\infty,X})$ is a Banach space, where $\|.\|_{\infty,X}$ is the supremum norm (see \cite[Theorem 5.9]{Dia}). As well as AAP- functional space, we have the following decomposition (see also \cite{Dia}): 
$$PAP(\mathbb{R},X) = AP(\mathbb{R},X) \oplus PAP_0(\mathbb{R},X).$$

The notion of pseudo almost periodic function is a generalisation of the periodic and almost periodic  functions. Precisely, we have the following inclusions
$$P(\mathbb{R},X) \hookrightarrow AP(\mathbb{R},X) \hookrightarrow PAP(\mathbb{R},X) \hookrightarrow C_b(\mathbb{R},X).$$
where $P(\mathbb{R},X)$ is the space of all  continuous and periodic functions from $\mathbb{R}$  to $X$.\\
{\bf Example.}
The function $h(t)=\sin{ t}+\sin({\sqrt{2}t})$ is almost periodic but not periodic,  $\tilde{h}(t) =\sin{ t}+\sin({\sqrt{2}t})+ e^{-|t|}$  is pseudo almost periodic but not almost periodic. Moreover, let $X$ be a Banach space and $g\in X - \left\{ 0\right\}$, we have that $f=hg\in AP(\mathbb{R},X)$ and $\tilde{f}=\tilde{h}g \in PAP(\mathbb{R},X)$.

\section{Linear systems: well-posedness of PAP-mild solutions}\label{S3}
In this section, we concentrate on studying the inhomogeneous linear system corresponding to system \eqref{KS}. 
\begin{equation}\label{KSlinear} 
\left\{
  \begin{array}{rll}
u_t \!\! &= \Delta u - \nabla \cdot (\omega\nabla (-\Delta+ \gamma I)^{-1})\omega + \dive f(t) \quad  & (t,x)\in \r\times \Omega, \hfill \cr
\nabla u \cdot \nu \!\!&=0 \quad & (t,x) \in \r\times \partial\Omega.
\end{array}\right.
\end{equation}
for a given $\omega$.
By Duhamel’s principle, we can define the mild solution of system \eqref{KSlinear} as a bounded solution of the following integral equation
\begin{equation}\label{00inteLinear}
u(t) = \int_{-\infty}^t \nabla \cdot e^{(t-s)\Delta}\left[ - \omega\nabla(-\Delta+ \gamma I)^{-1}\omega + f \right](s)ds.
\end{equation}
Setting $L_j = \partial_j(- \Delta + \gamma I)^{-1}$, the properties of this operator is given in the following lemma (see \cite[Lemma 4.1]{Fe2021}):
\begin{lemma}\label{invertEs}
Let $\Omega \subset \mathbb{R}^n$, $\gamma\geqslant 0$, $n\geqslant 2$, $1<p<n$ and $\frac{1}{q}=\frac{1}{p}-\frac{1}{n}$.
The operator $L_j$ is continuous from $L^{p}(\Omega)$ to $L^{q}(\Omega)$, for each $j=1,2...n$. Moreovver, there exists a constant $C>0$ independent of $f$ and $\gamma$ satisfying 
\begin{equation}
\norm{L_jf}_{L^{q}(\Omega)} \leqslant C{k(\gamma)}\norm{f}_{L^{p}(\Omega)},
\end{equation}
where ${k(0)}=1$ and ${k(\gamma)}=\gamma^{-(n-1)}$ if $\gamma>0$.
\end{lemma}
\begin{remark}
This lemma is also valid for the case of hyperbolic manifolds (see \cite[Lemma 3.3]{Pi}).
\end{remark}

The existence and uniqueness of  the bounded mild solutions of the inhomogeneous linear system \eqref{KSlinear} is established in the following lemma.
 
\begin{theorem}\label{Thm:linear}
Let $n\geq 2$ and ${\max \left\lbrace 3,n\right\rbrace} <p<2n$, the following assertions holds.
\begin{itemize}
\item[$(i)$] For given functions $\omega\in C_b(\mathbb{R}, L^{\frac{p}{2}}(\Omega))$ and $g \in C_b(\r, L^{\frac{p}{3}}(\Omega))$, there exists a unique mild solution of Equation \eqref{KSlinear} satisfying the integral equation \eqref{00inteLinear}. Moreover, the following boundedness holds
\begin{equation}\label{boundedness12}
\|u(t)\|_{L^{\frac{p}{2}}(\Omega)} \leq  C\left( k(\gamma)\norm{\omega}^2_{\infty,L^{\frac{p}{2}}(\Omega)} + \norm{f}_{\infty,L^{\frac{p}{3}}(\Omega)} \right).
\end{equation}
\item[$(ii)$] For given functions $(\omega,f) \in PAP(\mathbb{R}, L^{\frac{p}{2}}(\Omega)\times L^{\frac{p}{3}}(\Omega))$, there exists a unique PAP-mild solution of system \eqref{KSlinear} satisfying the integral equation \eqref{00inteLinear}.
\end{itemize}
\end{theorem}

\begin{proof} 
\item[$(i)$] By using Lemma \ref{Heatestimates} and Lemma \ref{invertEs} with noting that $1<\dfrac{p}{2}<n$, we can estimate
\begin{eqnarray*}
\norm{u(t)}_{L^{\frac{p}{2}}(\Omega)} &\leq&   \int_{-\infty}^t \norm{\nabla \cdot e^{(t-s) {\Delta}}\left[\omega\nabla(- {\Delta} + \gamma I)^{-1}\omega \right](s)}_{L^{\frac{p}{2}}(\Omega)}ds + \int_{-\infty}^t \norm{\nabla \cdot e^{(t-s) {\Delta}}f(s)}_{L^{\frac{p}{2}}(\Omega)}ds\cr
&\leq&  \int_{-\infty}^t k_2(1+(t-s)^{-\frac{1}{2}-\frac{n}{2}\left( \frac{4n-p}{pn}-\frac{2}{p}\right)})e^{-(t-s)\lambda_1}\norm{\left[ \omega\nabla(- {\Delta} + \gamma I)^{-1}\omega \right](s)}_{L^{\frac{pn}{4n-p}}(\Omega)}ds\cr
&&+ \int_{-\infty}^t k_2(1+(t-s)^{-\frac{n}{2}\left(\frac{1}{p}+\frac{1}{n}\right)})e^{-(t-s)\lambda_1}\norm{f(s)}_{L^{\frac{p}{3}}(\Omega)}ds\cr
&\leq&  \int_{-\infty}^t k_2(1+(t-s)^{-\frac{n}{p}})e^{-(t-s)\lambda_1} \norm{\omega(s)}_{L^{\frac{p}{2}}(\Omega)}  \norm{\left[\nabla(- {\Delta} + \gamma I)^{-1}\omega \right](s)}_{L^{\frac{pn}{2n-p}}(\Omega)}ds\cr
&&+ \int_{-\infty}^t k_2(1+(t-s)^{\frac{p+n}{2p}})e^{-(t-s)\lambda_1}\norm{f(s)}_{L^{\frac{p}{3}}(\Omega)}ds\cr
&\leqslant& \int_{-\infty}^t \left( 1+ (t-s)^{-\frac{n}{p}} \right) e^{-(t-s)\lambda_1}ds \left( Ck_2 {k(\gamma)}\norm{\omega}^2_{\infty, L^{\frac{p}{2}}(\Omega))} \right)\cr
&&+ \int_{-\infty}^t \left( 1+ (t-s)^{-\frac{p+n}{2p}} \right) e^{-(t-s)\lambda_1}ds \left(k_2 \norm{f}_{\infty,L^{\frac{p}{3}}(\Omega)}\right)\cr
&\leq& \int_0^{+\infty} \left( 1+ z^{-\frac{n}{p}} \right) e^{-z\lambda_1}dz \left( Ck_2 {k(\gamma)}\norm{\omega}^2_{\infty, L^{\frac{p}{2}}(\Omega))} \right) \cr
&&+ \int_0^{+\infty} \left( 1+ z^{-\frac{p+n}{2p}} \right) e^{-z\lambda_1}dz \left(k_2 \norm{f}_{\infty,L^{\frac{p}{3}}(\Omega)}\right)\cr
&\leq&  \left( \frac{1}{\lambda_1^{1- \frac{n}{p}}}\Gamma\left(1- \frac{n}{p} \right) + \frac{1}{\lambda_1} \right)\left(Ck_2{k(\gamma)}  \norm{\omega}^2_{\infty, L^{\frac{p}{2}}(\Omega)} \right)\cr
&&+  \left(\frac{1}{\lambda_1^{\frac{1}{2} - \frac{n}{2p}}}\Gamma\left( \frac{p-n}{2p} \right) + \frac{1}{\lambda_1} \right) k_2 \norm{f}_{\infty,L^{\frac{p}{3}}(\Omega)} \cr
&\leqslant& \widetilde{K}\left({k(\gamma)}   \norm{\omega}^2_{\infty,L^{\frac{p}{2}}(\Omega)} + \norm{f}_{\infty,L^{\frac{p}{3}}(\Omega)}\right),
\end{eqnarray*}
where  
$$\widetilde{K}=\max\left\{ Ck_2\left( \frac{1}{\lambda_1^{1- \frac{n}{p}}}\Gamma\left(1- \frac{n}{p} \right) + \frac{1}{\lambda_1} \right), k_2\left( \frac{1}{\lambda_1^{\frac{p-n}{2p}}}\Gamma\left( \frac{p-n}{2p} \right) + \frac{1}{\lambda_1}\right)\right\},$$ and  $\mathit{\mathbf{\Gamma}}$ means the Gamma function. Therefore, equation \eqref{00inteLinear} has a bounded solution which is a mild solution of linear system \eqref{KSlinear}. The uniqueness holds clearly.\\

\item[$(ii)$] From the well-posedness of integral equation \eqref{00inteLinear}, we can define the solution operator $S: C_b(\mathbb{R}, L^{\frac{p}{2}(\Omega)}\times L^{\frac{p}{3}}(\Omega)) \to C_b(\mathbb{R}, L^{\frac{p}{2}}(\Omega))$ as follows
\begin{equation}
S(\omega,f)(t) = u(t), \,\, t\in \mathbb{R},
\end{equation}
that is mild solution of \eqref{00inteLinear}. We now prove that if $(\omega,f) \in PAP(\mathbb{R}, L^{\frac{p}{2}}(\Omega)\times L^{\frac{p}{3}}(\Omega))$, then $S(\omega,f) \in PAP(\mathbb{R}, L^{\frac{p}{2}}(\Omega))$. Indeed, there exists two function $(\omega_1,f_1)\in AP(\mathbb{R}, L^{\frac{p}{2}}(\Omega)\times L^{\frac{p}{3}}(\Omega))$ and $(\omega_2,f_2)\in PAP_0(\mathbb{R}, L^{\frac{p}{2}}(\Omega)\times L^{\frac{p}{3}}(\Omega))$ satisfying that 
\begin{equation}
\omega=\omega_1+\omega_2 \hbox{   and   } f=f_1+f_2.
\end{equation}
Therefore, we can seperate $S(\omega,f)(t)$ as follows
\begin{eqnarray}\label{seperate}
S(\omega,f)(t) &=&  \int_{-\infty}^t \nabla \cdot e^{(t-s)\Delta}\left[ - \omega_1\nabla(-\Delta+ \gamma I)^{-1}\omega_1 + f_1 \right](s)ds\cr
&&+ \int_{-\infty}^t \nabla \cdot e^{(t-s)\Delta}\left[ - \omega_1\nabla(-\Delta+ \gamma I)^{-1}\omega_2 + f_2 \right](s)ds\cr
&&+ \int_{-\infty}^t \nabla \cdot e^{(t-s)\Delta}\left[ - \omega_2\nabla(-\Delta+ \gamma I)^{-1}\omega_1 - \omega_2\nabla(-\Delta+ \gamma I)^{-1}\omega_2\right](s)ds\cr
&=& S(\omega_1,f_1)(t) + \mathbb{S}(t).
\end{eqnarray}
To obtain $S(\omega,f) \in PAP(\mathbb{R}, L^{\frac{p}{2}}(\Omega))$, we prove that
\begin{equation}\label{1AP}
S(\omega_1,f_1) \in AP(\mathbb{R}, L^{\frac{p}{2}}(\Omega))
\end{equation}
and
\begin{equation}\label{2PAP}
\mathbb{S} \in PAP_0(\mathbb{R}, L^{\frac{p}{2}}(\Omega)).
\end{equation}
Indeed, since $(\omega_1,f_1)\in AP(\mathbb{R}, L^{\frac{p}{2}}(\Omega)\times L^{\frac{p}{3}}(\Omega))$, we have that: for each $ \varepsilon  > 0$, there exists $l_{\epsilon}>0 $ such that every interval of length $l_{\epsilon}$ contains at least a number $T $ with the following property
\begin{equation}\label{APine}
 \sup_{t \in \mathbb{R} } \| \omega_1(t+T)  - \omega_1(t) \|_{L^{\frac{p}{2}}(\Omega)} + \| f_1(t+T)-f_1(t)\|_{L^{\frac{p}{3}}(\Omega)} < \varepsilon.
\end{equation}
Moreover, by changing variable we can express
\begin{eqnarray}\label{APestimate}
&&S(\omega_1,f_1)(t+T) - S(\omega_1,f_1)(t) \cr
&=&  \int_{-\infty}^{t+T} \nabla \cdot e^{(t+T-s)\Delta}\left[ - \omega_1\nabla(-\Delta+ \gamma I)^{-1}\omega_1 + f_1 \right](s)ds\cr
&&- \int_{-\infty}^{t} \nabla \cdot e^{(t-s)\Delta}\left[ - \omega_1\nabla(-\Delta+ \gamma I)^{-1}\omega_1 + f_1 \right](s)ds\cr
&=&\int_{-\infty}^{t} \nabla \cdot e^{(t-s)\Delta}\left[ - \omega_1\nabla(-\Delta+ \gamma I)^{-1}\omega_1 + f_1 \right](s+T)ds\cr
&&- \int_{-\infty}^{t} \nabla \cdot e^{(t-s)\Delta}\left[ - \omega_1\nabla(-\Delta+ \gamma I)^{-1}\omega_1 + f_1 \right](s)ds\cr
&=&\int_{-\infty}^{t} \nabla \cdot e^{(t-s)\Delta}\left[ (- \omega_1\nabla(-\Delta+ \gamma I)^{-1}\omega_1)(s+T) + (\omega_1\nabla(-\Delta+ \gamma I)^{-1}\omega_1)(s) \right]ds\cr
&&- \int_{-\infty}^{t} \nabla \cdot e^{(t-s)\Delta}\left[ f_1(s+T)- f_1(s) \right] ds\cr
&=&\int_{-\infty}^{t} \nabla \cdot e^{(t-s)\Delta}\left[ (- \omega_1(s+T)+\omega_1(s))\nabla(-\Delta+ \gamma I)^{-1}\omega_1)(s+T)\right]ds\cr
&& + \int_{-\infty}^{t} \nabla \cdot e^{(t-s)\Delta} \left[(\omega_1(s)\nabla(-\Delta+ \gamma I)^{-1}(\omega_1(s)-\omega_1(s+T)) \right]ds\cr
&&- \int_{-\infty}^{t} \nabla \cdot e^{(t-s)\Delta}\left[ f_1(s+T)- f_1(s) \right] ds.
\end{eqnarray}
Hence, by the same estimations as in Assertion $(i)$ and inequality \eqref{APine}  we obtain from \eqref{APestimate} that
\begin{eqnarray*}
\norm{S(\omega_1,f_1)(t+T) - S(\omega_1,f_1)(t)}_{L^{\frac{p}{2}}(\Omega)} 
&\leqslant& 2\widetilde{K}k(\gamma)\| \omega_1\|_{\infty,L^{\frac{p}{2}}(\Omega)}\|\omega(\cdot+T)-\omega(\cdot) \|_{\infty,L^{\frac{p}{2}}(\Omega)}\cr
&&+ \widetilde{K}\| f(\cdot+T)-f(\cdot)\|_{\infty,L^{\frac{p}{3}}(\Omega)}\cr
&\leqslant& \widetilde{K}(2k(\gamma)+1)\varepsilon.
\end{eqnarray*}
This implies $S(\omega_1,f_1)\in AP(\mathbb{R},L^{\frac{p}{2}}(\Omega))$ and \eqref{1AP} holds.

We remain to prove \eqref{2PAP} which is equivalent to
\begin{equation}\label{3PAP}
\lim_{L\to +\infty} \frac{1}{2L}\int_{-L}^L \| \mathbb{S}(t)\|_{L^{\frac{p}{2}}(\Omega)}dt = 0.
\end{equation}
From \eqref{seperate} we have
\begin{eqnarray}\label{4PAP}
\| \mathbb{S}(t)\|_{L^{\frac{p}{2}}(\Omega)} &\leqslant& \int_{-\infty}^t \| \nabla \cdot e^{(t-s)\Delta}\left[\omega_1\nabla(-\Delta+ \gamma I)^{-1}\omega_2 \right](s)\|_{L^{\frac{p}{2}}(\Omega)}ds\cr
&&+ \int_{-\infty}^t \|\nabla \cdot e^{(t-s)\Delta}\left[ \omega_2\nabla(-\Delta+ \gamma I)^{-1}\omega_1\right](s) \|_{L^{\frac{p}{2}}(\Omega)} ds\cr
&&+ \int_{-\infty}^t \|\nabla \cdot e^{(t-s)\Delta}\left[ \omega_2\nabla(-\Delta+ \gamma I)^{-1}\omega_2\right](s) \|_{L^{\frac{p}{2}}(\Omega)} ds\cr
&&+\int_{-\infty}^t \|\nabla \cdot e^{(t-s)\Delta}f_2(s) \|_{L^{\frac{p}{2}}(\Omega)} ds.
\end{eqnarray}
We first prove that 
\begin{equation}\label{161024-n1}
\lim_{L\to +\infty} \frac{1}{2L}\int_{-L}^L \left(\varphi(t) + \psi(t) \right)dt = 0
\end{equation}
where, 
$$\varphi(t)=\int_{-L}^t \| \nabla \cdot e^{(t-s)\Delta}\left[\omega_1\nabla(-\Delta+ \gamma I)^{-1}\omega_2 \right](s)\|_{L^{\frac{p}{2}}(\Omega)}ds$$
and 
$$\psi(t)=\int_{-\infty}^{-L} \| \nabla \cdot e^{(t-s)\Delta}\left[\omega_1\nabla(-\Delta+ \gamma I)^{-1}\omega_2 \right](s)\|_{L^{\frac{p}{2}}(\Omega)}ds$$ 
provided that $\omega_2 \in PAP_0(\mathbb{R}, L^{\frac{p}{2}}(\Omega))$. Indeed, by similar arguments as in Assertion $(i)$ we have
\begin{align*}
\varphi(t) &\leq Ck_2k(\gamma)\int_{-L}^t\Big(1+(t-s)^{-\frac{n}{p}}\Big)e^{-(t-s)\lambda_1}\|w_1(s)\|_{L^{\frac{p}{2}}(\Omega)}\|w_2(s)\|_{L^{\frac{p}{2}}(\Omega)}ds\\
&\leq Ck_2k(\gamma)\|w_1\|_{\infty,L^{\frac{p}{2}}(\Omega)}\int_{-L}^t\Big(1+(t-s)^{-\frac{n}{p}}\Big)e^{-(t-s)\lambda_1}\|w_2(s)\|_{L^{\frac{p}{2}}(\Omega)}ds\\
&=C_1\int_{-L}^t\Big(1+(t-s)^{-\frac{n}{p}}\Big)e^{-(t-s)\lambda_1}\|w_2(s)\|_{L^{\frac{p}{2}}(\Omega)}ds\\
&=C_1\int_0^{t+L}\Big(1+s^{-\frac{n}{p}}\Big)e^{-s\lambda_1}\|w_2(t-s)\|_{L^{\frac{p}{2}}(\Omega)}ds,
\end{align*}
where $C_1=Ck_2k(\gamma)\|w_1\|_{\infty,L^{\frac{p}{2}}(\Omega)}$. Hence,
\begin{align*}
0&\leq \dfrac{1}{2L}\int_{-L}^L\varphi(t)dt \leq C_1\dfrac{1}{2L}\int_{-L}^L\left(\int_0^{t+L}\Big(1+s^{-\frac{n}{p}}\Big)e^{-s\lambda_1}\|w_2(t-s)\|_{L^{\frac{p}{2}}(\Omega)}ds\right)dt\\
&=C_1\dfrac{1}{2L}\int_0^{2L}\left(\Big(1+s^{-\frac{n}{p}}\Big)e^{-s\lambda_1}\int_{s}^L \|w_2(t-s)\|_{L^{\frac{p}{2}}(\Omega)}dt\right)ds\\
&=C_1\int_0^{2L}\Big(1+s^{-\frac{n}{p}}\Big)e^{-s\lambda_1}\dfrac{1}{2L}\int_{-L}^{t} \|w_2(z)\|_{L^{\frac{p}{2}}(\Omega)}dz ds\\
&\leq C_1\int_0^{+\infty}\phi_L(s)ds,
\end{align*}
where
$$\phi_L(s) = \Big(1+s^{-\frac{n}{p}}\Big)e^{-s\lambda_1}\dfrac{1}{2L}\int_{-L}^{L} \|w_2(t)\|_{L^{\frac{p}{2}}(\Omega)}dt.$$
Since $w_2\in PAP_0(\mathbb{R},L^{\frac{p}{2}}(\Omega))$, we imply that
$$\lim_{L\to\infty}\phi_L(s) = 0\;\;\forall s>0.$$
Besides,
$$\phi_L(s)\leq \Big(1+s^{-\frac{n}{p}}\Big)e^{-s\lambda_1}\|w_2\|_{\infty,L^{\frac{p}{2}}(\Omega)},$$
and
$$\int_0^{+\infty}\Big(1+s^{-\frac{n}{p}}\Big)e^{-s\lambda_1} ds = \dfrac{1}{\lambda_1}+\dfrac{1}{\lambda_1^{(1-\frac{n}{p})}}\Gamma\left(1-\frac{n}{p}\right).$$
Hence, by the Lebesgue dominated convergence, we get
\begin{equation*}
\lim_{L\to +\infty} \int_0^{+\infty}\phi_L(s)ds = \int_0^{+\infty}(1+s^{-\frac{n}{p}})e^{-s\lambda_1}ds \lim_{L\to +\infty}\frac{1}{2L}\int_{-L}^L\| \omega_2(t)\|_{L^{\frac{p}{2}}(\Omega)}dt =0.
\end{equation*}
This leads to
\begin{equation}\label{pap1}
\lim_{L\to +\infty}\frac{1}{2L}\int_{-L}^L\phi(t)dt=0.
\end{equation}
Since we have the following boundedness (the proof is based on the same estimates as in Assertion $(i)$):
\begin{equation*}
\int_{-\infty}^t  \| \nabla \cdot e^{(t-s)\Delta}\left[\omega_1\nabla(-\Delta+ \gamma I)^{-1}\omega_2 \right](s)\|_{L^{\frac{p}{2}}(\Omega)}ds < \widetilde{K}k(\gamma)\| \omega_1\|_{\infty,L^{\frac{p}{2}}(\Omega)}\| \omega_2\|_{\infty,L^{\frac{p}{2}}(\Omega)}<+\infty.
\end{equation*}
Hence, we have clearly
\begin{equation*}
\lim_{L\to +\infty}\int_{-\infty}^{-L}  \| \nabla \cdot e^{(t-s)\Delta}\left[\omega_1\nabla(-\Delta+ \gamma I)^{-1}\omega_2 \right](s)\|_{L^{\frac{p}{2}}(\Omega)}ds = \lim_{L\to +\infty}\psi(t)=0.
\end{equation*}
This leads to
\begin{equation}\label{pap2}
\lim_{L\to +\infty}\frac{1}{2L}\int_{-L}^L\psi(t)dt=0.
\end{equation}
Combining \eqref{pap1} and \eqref{pap2} we get the desired limit \eqref{161024-n1}.
The same limits hold for the rest terms in righ hand-side of \eqref{4PAP} and we obtain \eqref{3PAP}. Our proof is complete.
\end{proof}

\section{Semi-linear systems: well-posedness and exponential stability} \label{S4}
The aim of this section is to prove the existence, uniqueness of the pseudo almost periodic mild solutions for system \eqref{KS}. We state and prove the main results of this section in the following lemma. 
\begin{theorem}\label{thm2.20}
Let $n\geq 2$ and ${\max \left\lbrace 3,n\right\rbrace }<p<2n$. Suppose that a function $f$  belongs  to $PAP(\mathbb{R}, L^{\frac{p}{3}}(\Omega))$.  If the norm $\|f\|_{\infty,L^{\frac{p}{3}}(\Omega)}$ is sufficiently small,  the system \eqref{KS} has a unique PAP-mild solution $\hat{u}$ on a small ball of  
$C_b(\r, L^{\frac{p}{2}}(\Omega))$. Moreover, the solution $\hat{u}$ is exponential stable in the sense that: for any other mild solution $u$ of \eqref{KS} satisfying that $\|u\|_{\infty,L^{\frac{p}{2}}(\Omega)}$ is small enough and for $0<\sigma<\lambda_1$, we have
\begin{equation}\label{condition}
\lim_{t\to +\infty} e^{\sigma t}\norm{e^{t\Delta}(\hat{u}(0)- u(0))}_{L^{\frac{p}{2}}(\Omega)} = 0
\end{equation}
if and only if
\begin{equation}\label{exponentialdecay}
\lim_{t\to +\infty} e^{\sigma t}\norm{\hat{u}(t)- u(t)}_{L^{\frac{p}{2}}(\Omega)} = 0
\end{equation} 
\end{theorem}
\begin{proof}
Let 
\begin{eqnarray}\label{bro}
\B_\rho^{PAP}=\left\{\omega\in PAP(\r, L^{\frac{p}{2}}(\Omega)):  \norm{\omega}_{{\infty,L^{\frac{p}{2}}(\Omega}} \le \rho \right\}
\end{eqnarray}
be a ball centered at zero and radius $\rho>0$. 

For a given function $\omega\in \B_\rho^{PAP}$, we consider the following linear integral equation 
\begin{equation}\label{ns1}
u(t) =\int_{-\infty}^t \nabla_x \cdot e^{(t-s) {\Delta}}\left[ -\omega\nabla_x(- {\Delta}+ \gamma I)^{-1}\omega + f \right](s)ds.
\end{equation} 
By Theorem \ref{Thm:linear} $(i)$, integral equation \eqref{ns1} has a unique pseudo almost periodic mild solution $u$ satisfying
\begin{eqnarray}\label{CoreEstimate}
\norm{u(t)}_{L^{\frac{p}{2}}(\Omega)}& \leq&    \widetilde{K}\left({k(\gamma)}\norm{\omega}^2_{\infty,L^{\frac{p}{2}}(\Omega)} + \norm{f}_{\infty,L^{\frac{p}{3}}(\Omega)} \right)\cr
& \leq&   \widetilde{K}\left( {k(\gamma)}\rho^2 + \norm{f}_{\infty,L^{\frac{p}{3}}(\Omega)} \right)\cr
&\leq& {\rho}
\end{eqnarray}
provided that $ \rho$ and $\norm{f}_{\infty,L^{\frac{p}{3}}(\Omega)}$ are small enough. Therefore, we can define a map from $\mathcal{B}^{PAP}_\rho$ into itself as follows
\begin{equation}\label{defphi}
\begin{split}
\Phi:  \mathcal{B}_\rho^{PAP} &\to \mathcal{B}^{PAP}_\rho \cr 
\omega &\mapsto \Phi(\omega)=u
\end{split}
\end{equation}
where $u$ is a unique solution of \eqref{ns1}. Clearly, it turns out that 
\begin{equation}\label{defphi1}
\Phi(\omega)(t) = \int_{-\infty}^t \nabla_x \cdot e^{(t-s) {\Delta}}\left[ -\omega\nabla_x(- {\Delta}+ \gamma I)^{-1}\omega + f \right](s)ds.
\end{equation}
Therefore, we have that for $\omega_1, \omega_2\in \B_\rho^{PAP}$, the function $u:=\Phi(\omega_1)-\Phi(\omega_2)$ becomes a unique pseudo almost periodic mild solution to the equation 
\begin{eqnarray*}
\partial_tu  - \Delta u   &=& - \omega_1\nabla_x(- {\Delta}+ \gamma I)^{-1}\omega_1 + \omega_2\nabla_x(- {\Delta}+ \gamma I)^{-1}\omega_2 \cr
&=&{- \omega_1\nabla_x(- {\Delta}+ \gamma I)^{-1}(\omega_1-\omega_2) + (\omega_2-\omega_1)\nabla_x(- {\Delta}+ \gamma I)^{-1}\omega_2.}
 \end{eqnarray*}
Thus, by \eqref{defphi1} and the same way to establish inequality \eqref{CoreEstimate}, we can estimate
\begin{eqnarray}\label{Core}
\norm{\Phi(\omega_1)-\Phi(\omega_2)}_{\infty,L^{\frac{p}{2}}(\Omega)} &\leqslant 2\widetilde{K} {k(\gamma)} \rho \norm{\omega_1-\omega_2}_{\infty, L^{\frac{p}{2}}(\Omega)}.
\end{eqnarray} 
This shows that the map  $\Phi$ invokes a contradiction if $\rho$ is sufficiently small.

By applying fixed point arguments, there is a unique fixed point $\hat{u}$ of $\Phi$. Due to the definition of $\Phi$, this function $\hat{u}$ is a PAP-solution of semilinear integral equation \eqref{00inte} which is a PAP-mild solution of Keller-Segel system \eqref{KS}. The uniqueness of $\hat{u}$ in the small ball $\B_\rho^{PAP}$ is clearly by using \eqref{Core}.

Now we prove the exponential decay of pseudo almost periodic solution $\hat{u}$. First, we can rewrite for $t>0$ the mild solutions as
\begin{eqnarray}
\hat{u}(t)&=& \int_{-\infty}^0 \nabla_x \cdot e^{(t-s) {\Delta}}\left[ -\hat{u}\nabla_x(- {\Delta}+ \gamma I)^{-1}\hat{u} + f \right](s)ds\cr
&& + \int_0^t \nabla_x \cdot e^{(t-s) {\Delta}}\left[ -\hat{u}\nabla_x(- {\Delta}+ \gamma I)^{-1}\hat{u} + f \right](s)ds\cr
&=& e^{t\Delta}\hat{u}(0) + \int_0^t \nabla_x \cdot e^{(t-s) {\Delta}}\left[ -\hat{u}\nabla_x(- {\Delta}+ \gamma I)^{-1}\hat{u} + f \right](s)ds 
\end{eqnarray}
and 
\begin{eqnarray}
{u}(t)&=&  e^{t\Delta}{u}(0) + \int_0^t \nabla_x \cdot e^{(t-s) {\Delta}}\left[ -{u}\nabla_x(- {\Delta}+ \gamma I)^{-1}{u} + f \right](s)ds. 
\end{eqnarray}
Therefore, for positive constant $\tilde{\rho}$ such that $\|u \|_{C_b(\mathbb{R},L^{\frac{n}{2}}(\Omega))}<\tilde{\rho} $, we obtain that
\begin{eqnarray}\label{Gron}
&&\norm{\hat{u}(t) - u(t)}_{L^{\frac{p}{2}}(\Omega)} \cr
&\leqslant& \|e^{t\Delta}(\hat{u}(0)-u(0)) \|_{L^{\frac{p}{2}}(\Omega)} \cr
&&+ \int_0^t \norm{\nabla_x \cdot e^{(t-s) {\Delta}}\left[ \hat{u}\nabla_x(- {\Delta} + \gamma I)^{-1} \hat{u} - {u}\nabla_x(- {\Delta} + \gamma I)^{-1} {u} \right](s)}_{L^{\frac{p}{2}}(\Omega)}ds\cr
&\leq& \|e^{t\Delta}(\hat{u}(0)-u(0)) \|_{L^{\frac{p}{2}}(\Omega)} \cr
&&+\int_0^t k_2(1+(t-s)^{-\frac{n}{p}})e^{-(t-s)\lambda_1}\norm{\left[ \hat u\nabla_x(- {\Delta_{\Omega}} + \gamma I)^{-1}\hat u - u\nabla_x(- {\Delta_{\Omega}} + \gamma I)^{-1}u \right](s)}_{L^{\frac{pn}{4n-p}}(\Omega)}ds\cr
&\leq& \|e^{t\Delta}(\hat{u}(0)-u(0)) \|_{L^{\frac{p}{2}}(\Omega)} \cr
&&+ \int_0^t k_2(1+(t-s)^{-\frac{n}{p}}) e^{-(t-s)\lambda_1}\norm{\hat u(s)-u(s)}_{L^{\frac{p}{2}}(\Omega)}  \norm{\left[\nabla_x(- {\Delta} + \gamma I)^{-1}\hat u \right](s)}_{L^{\frac{pn}{2n-p}}(\Omega)}ds\cr
&&+ \int_0^t k_2(1+(t-s)^{-\frac{n}{p}}) e^{-(t-s)\lambda_1}\norm{ u(s)}_{L^{\frac{p}{2}}(\Omega)}  \norm{\left[\nabla_x(- {\Delta} + \gamma I)^{-1}(\hat u -u) \right](s)}_{L^{\frac{pn}{2n-p}}(\Omega)}ds\cr
&\leq& \|e^{t\Delta}(\hat{u}(0)-u(0)) \|_{L^{\frac{p}{2}}(\Omega)} \cr
&&+ k_2{k(\gamma)} \int_{-\infty}^t \left( 1+ (t-s)^{-\frac{n}{p}} \right) e^{-(t-s)\lambda_1} \norm{\hat u(s) - u(s)}_{ L^{\frac{p}{2}}(\Omega)}ds \norm{\hat u}_{C_b(\r, L^{\frac{p}{2}}(\Omega))}\cr
&&+ k_2{k(\gamma)} \int_{-\infty}^t \left( 1+ (t-s)^{-\frac{n}{p}} \right) e^{-(t-s)\lambda_1} \norm{\hat u(s) - u(s)}_{ L^{\frac{p}{2}}(\Omega)}ds \norm{ u}_{C_b(\r, L^{\frac{p}{2}}(\Omega))}\cr
&\leq& \|e^{t\Delta}(\hat{u}(0)-u(0)) \|_{L^{\frac{p}{2}}(\Omega)} \cr
&&+ k_2{k(\gamma)} (\rho+\tilde{\rho}) \int_0^t \left( 1+(t-s)^{-\frac{n}{p}} \right) e^{-(t-s)\lambda_1} \norm{\hat u(s)-u(s)}_{ L^{\frac{p}{2}}(\Omega))}ds.
\end{eqnarray}
This leads to
\begin{eqnarray}\label{Gron'}
&&e^{\sigma t}\norm{\hat{u}(t) - u(t)}_{L^{\frac{p}{2}}(\Omega)} \cr
&\leq& e^{\sigma t}\|e^{t\Delta}(\hat{u}(0)-u(0)) \|_{L^{\frac{p}{2}}(\Omega)} \cr
&&+ k_2{k(\gamma)} (\rho+\tilde{\rho}) \int_0^t \left( 1+(t-s)^{-\frac{n}{p}} \right) e^{-(t-s)(\lambda_1-\sigma)}ds  \sup_{t>0}e^{\sigma t}\norm{\hat u(t)-u(t)}_{ L^{\frac{p}{2}}(\Omega))}\cr
&\leq& e^{\sigma t}\|e^{t\Delta}(\hat{u}(0)-u(0)) \|_{L^{\frac{p}{2}}(\Omega)} + k_2{k(\gamma)} L(\rho+\tilde{\rho})  \sup_{t>0}e^{\sigma t}\norm{\hat u(t)-u(t)}_{ L^{\frac{p}{2}}(\Omega))},
\end{eqnarray}
where
$$\int_0^t \left( 1+(t-s)^{-\frac{n}{p}} \right) e^{-(t-s)(\lambda_1-\sigma)}ds < \frac{1}{\lambda_1-\sigma} + \frac{1}{(\lambda_1-\sigma)^{1-\frac{n}{p}}}\Gamma\left(  1-\frac{n}{p}\right).$$
Now, assume that the limit \eqref{condition} holds, then we have from \eqref{Gron'} that
\begin{equation}
(1-k_2k(\gamma)L(\rho+\tilde{\rho}))\limsup_{t\to +\infty}e^{\sigma t}\norm{\hat u(t)-u(t)}_{ L^{\frac{p}{2}}(\Omega))}=0.
\end{equation}
If $\rho$ and $\tilde{\rho}$ are small enough, then $1-k_2k(\gamma)L(\rho+\tilde{\rho})>0$ and we get the limit \eqref{exponentialdecay}.

On the other hand, by the same way as \eqref{Gron} and \eqref{Gron} we can estimate that
\begin{eqnarray*}
e^{\sigma t}\|e^{t\Delta}(\hat{u}(0)-u(0)) \|_{L^{\frac{p}{2}}(\Omega)}&\leq&
(1+k_2{k(\gamma)} L(\rho+\tilde{\rho}))  \sup_{t>0}e^{\sigma t}\norm{\hat u(t)-u(t)}_{ L^{\frac{p}{2}}(\Omega))}.
\end{eqnarray*}
Hence, if the limit \eqref{exponentialdecay} holds, then we have clearly that \eqref{condition} holds.
\end{proof}
\begin{remark}
We notice that the limit \eqref{condition} holds if we assume further that $\int_\Omega \hat{u}(0) d\mathrm{Vol}_\Omega = \int_\Omega {u}(0) d\mathrm{Vol}_\Omega =0$. Indeed, by using estimate \eqref{dispersive1} in Lemma \ref{Heatestimates} we have
\begin{equation*}
\|e^{t\Delta}(\hat{u}(0)-u(0)) \|_{L^{\frac{p}{2}}(\Omega)} \leq k_1e^{-\lambda_1 t} \|\hat{u}(0)-u(0) \|_{L^{\frac{p}{2}}(\Omega)}.
\end{equation*} 
This leads to the limit \eqref{condition} for $0<\sigma<\lambda_1$.
\end{remark}

\section{Results on real hyperbolic manifolds}\label{S5} 
In this section we provide the well-posedness and stability of PAP-mild solutions for Keller-Segel (P-E) systems on the real hyperbolic spaces. These results will very like the ones obtained on bounded domains of Euclidean spaces but it is interesting since we work on the whole hyperbolic space. In particular, let $(\mathbb{H}^{n},{\mathfrak g})= (\mathbb{H}^{n}(\mathbb{R}),{\mathfrak g})$ stand for a real hyperbolic manifold, where $n\geqslant 2$ is the dimension, endowed with a Riemannian metric ${\mathfrak g}$. This space is realized via a  hyperboloid in $\mathbb{R}^{n+1}$ by considering the upper sheet
$$
\left\{  (x_{0},x_{1},...,x_{n})\in\mathbb{R}^{n+1};\text{ }\,x_{0}\geq1\text{
and }x_{0}^{2}-x_{1}^{2}-x_{2}^{2}...-x_{n}^{2}=1\,\right\},$$
where the metric is given by $d{\mathfrak g} =-dx_{0}^{2}+dx_{1}^{2}+...+dx_{n}^{2}.$

In geodesic polar coordinates, the hyperbolic manifold $(\mathbb{H}^{n},{\mathfrak g})$ can be described as
$$
\mathbb{H}^{n}=\left\{  (\cosh\tau,\omega\sinh\tau),\,\tau\geq0,\omega
\in\mathbb{S}^{n-1}\right\}
$$
with $d{\mathfrak g}=d\tau^{2} +(\sinh\tau)^{2}d\omega^{2},$ where $d\omega^{2}$ is the
canonical metric on the sphere $\mathbb{S}^{n-1}$. In these coordinates, the
Laplace-Beltrami operator $\Delta_{\mathbb{H}^n}$ on $\mathbb{H}^{n}$ can be
expressed as
$$
\Delta_{\mathbb{H}^n}=\partial_{r}^{2}+(n-1)\coth r\partial
_{r}+\sinh^{-2}r\Delta_{\mathbb{S}^{n-1}}.
$$
It is well known that the spectrum of $-\Delta_{\mathbb{H}^n}$ is the half-line
$\left[\dfrac{(n-1)^2}{4},\infty \right)$. 

The dispersive and smoothing estimates of heat semigroup on hyperbolic space are well
studied in the literature for hyperbolic spaces. For convenience we recall the estimates established by Pierfelice (see \cite[Theorem 4.1 and Corollary 4.3]{Pi}):
\begin{lemma}\label{estimates}
\begin{itemize}
\item[(i)] For $t>0$, and $p$, $q$ such that $1\leq p \leq q \leq \infty$, 
the following dispersive estimate holds: 
\begin{equation}\label{dispersive}
\left\| e^{t {\Delta_{\mathbb{H}^n}}}u_0\right\|_{L^q(\mathbb{H}^n)} \leq [h_n(t)]^{\frac{1}{p}-\frac{1}{q}}e^{-t( \gamma_{p,q})}\left\|u_0\right\|_{L^p(\mathbb{H}^n)} 
\end{equation}
for all $u_0 \in L^p(\mathbb{H}^n,\mathbb{R})$, where 
 $$h_n(t) = \tilde{C}\max\left( \frac{1}{t^{n/2}},1 \right),\, 
   \gamma_{p,q}=\frac{\delta_n}{2}\left[ \left(\frac{1}{p} - \frac{1}{q} \right) + \frac{8}{q}\left( 1 - \frac{1}{p} \right) \right]$$ 
and $\delta_n$ is a positive constant depending only on $n$.  
\item[(ii)] For $t>0$, and $p,q$ such that $1\leqslant p\leqslant q \leqslant\infty$, the following estimate holds:
\begin{equation}
\left\|  {\nabla_x\cdot e^{t\Delta_{\mathbb{H}^n}}} V_0 \right\|_{L^q(\mathbb{H}^n)} \leqslant [h_n(t)]^{\frac{1}{p}-\frac{1}{q}+\frac{1}{n}}e^{-t\left( \frac{\gamma_{q,q}+\gamma_{p,q}}{2} \right)} \left\|V_0\right\|_{L^p(\mathbb{H}^n)}
\end{equation}
for all vector field $V_0 \in L^p(\mathbb{H}^n)$. The functions $h_n(t)$ and $\gamma_{p,q}$ are defined as in Assertion (i).
\end{itemize}
\end{lemma}
\begin{remark}
For dispersive estimate \eqref{dispersive} we need not the condition $\int_{\mathbb{H}^n}u_0d\mathrm{Vol}_{\mathbb{H}^n}=0$ as in the case of bounded domain $\Omega\subset \mathbb{R}^n$ (see \eqref{dispersive1}). This leads a slight difference in the stability condition in Theorem \ref{ThecaseHn} below (in comparing with the one obtained in Theorem \ref{thm2.20}).
\end{remark}

We now consider the parabolic-parabolic Keller–Segel system on $\mathbb{H}^n \,\, (n \geqslant 2)$ which is given as follows
\begin{equation}\label{KSH} 
\left\{
  \begin{array}{rll}
u_t \!\! &= \Delta_{\mathbb{H}^n} u - \nabla \cdot (u\nabla v) + \dive f(t) \quad  & (x,t)\in \mathbb{H}^n \times \mathbb{R}, \hfill \cr
- \Delta_{\mathbb{H}^n} v + \gamma v \!\!&=  u \quad & (x,t) \in \mathbb{H}^n \times \mathbb{R},
\end{array}\right.
\end{equation}
By Duhamel’s principle, the mild solution of system \eqref{KSH} is given by
\begin{equation}\label{inter1}
u(t) =  \int_{-\infty}^t \nabla \cdot e^{(t-s)\Delta_{\mathbb{H}^n}}\left[ - u\nabla(-\Delta_{\mathbb{H}^n} + \gamma I)^{-1}u + f\right](s)ds.
\end{equation}
The linear integral corresponding to \eqref{inter1} is
\begin{equation}\label{inter2}
u(t) =  \int_{-\infty}^t \nabla \cdot e^{(t-s)\Delta_{\mathbb{H}^n}}\left[ - \omega\nabla(-\Delta_{\mathbb{H}^n} + \gamma I)^{-1}\omega + f\right](s)ds.
\end{equation}
By similar way as in Sections \eqref{S3} and \eqref{S4}, we can obtain the existence, uniqueness and stability for Keller-Segel (P-E) in the following theorem.
\begin{theorem}\label{ThecaseHn}
Let $n\geq 2$, the following assertions holds
\begin{itemize}
\item[$(i)$] For given functions $\omega\in C_b(\mathbb{R}, L^{\frac{p}{2}}(\mathbb{H}^n))$ and $g \in C_b(\r, L^{\frac{p}{3}}(\mathbb{H}^n))$, there exists a unique bounded solution of the integral equation \eqref{inter2}. Moreover, the following boundedness holds
\begin{equation}
\|u(t)\|_{L^{\frac{p}{2}}(\mathbb{H}^n)} \leq  C\left( k(\gamma) \norm{\omega}^2_{\infty,L^{\frac{p}{2}}(\mathbb{H}^n)} + \norm{f}_{\infty,L^{\frac{p}{3}}(\mathbb{H}^n)} \right).
\end{equation}
\item[$(ii)$] For given functions $(\omega,f) \in PAP(\mathbb{R}, L^{\frac{p}{2}}(\mathbb{H}^n)\times L^{\frac{p}{3}}(\mathbb{H}^n))$, there exists a unique PAP-mild solution of integral equation \eqref{inter2}.
\item[$(iii)$] Assume that ${\max \left\lbrace 3,n\right\rbrace }<p<2n$ and the function $f \in PAP(\mathbb{R}, L^{\frac{p}{3}}(\mathbb{H}^n))$.  If the norm $\|f\|_{\infty,L^{\frac{p}{3}}(\mathbb{H}^n)}$ is sufficiently small,  the system \eqref{KSH} has a unique PAP-mild solution $\tilde{u}$ on a small ball of  
$C_b(\r, L^{\frac{p}{2}}(\mathbb{H}^n))$. Moreover, the solution $\tilde{u}$ is exponential stable in the sense that: for any other mild solution $u$ of \eqref{KSH}, we have
\begin{equation}
\norm{\tilde{u}(t)- u(t)}_{L^{\frac{p}{2}}(\mathbb{H}^n)} \leqslant De^{-\sigma t}\norm{\tilde{u}(0)-u(0)}_{L^{\frac{p}{2}}(\mathbb{H}^n)}
\end{equation}
for all $t>0$, where $\sigma= \min\left\{ \gamma_{p/2,p/2},\,\frac{\gamma_{p/2,p/2}+\gamma_{pn/(4n-p),p/2}}{2},\, \frac{\gamma_{p/2,p/2}+\gamma_{p/3,p/2}}{2}\right\}$.
\end{itemize}
\end{theorem}
\begin{proof}
The proof is similar Theorem \ref{Thm:linear} and Theorem \ref{thm2.20} with the same calucations provided in \cite[Section 3]{Xu}.
\end{proof}

\end{document}